\documentclass{amsart}

\usepackage{amsmath}
\usepackage{centernot}
\usepackage{mathtools}
\usepackage{amssymb}
\usepackage{graphicx}
\usepackage{amsfonts}
\usepackage{microtype} 

\usepackage{hyperref}
\usepackage{tikz-cd}
\usepackage{amsthm}
\usepackage[symbol]{footmisc}

\newtheorem{theorem}{Theorem}
\newtheorem{hyp}{Hypothesis}
\newtheorem{lemma}{Lemma}
\newtheorem{proposition}[theorem]{Proposition}

\newtheorem{definition}{Definition}

\theoremstyle{remark}
\newtheorem{remark}[lemma]{Remark}

\numberwithin{equation}{section}
\definecolor{oxfordblue}{rgb}{0.0, 0.13, 0.28}
\definecolor{phthalogreen}{rgb}{0.07, 0.21, 0.14}

\usepackage[utf8]{inputenc}
\usepackage[english]{babel}

\usepackage{hyperref}
\hypersetup{
colorlinks   = true, 
urlcolor     = oxfordblue, 
linkcolor    = oxfordblue, 
citecolor   = phthalogreen 
}

\newcommand{\mS}{\ensuremath{\mathcal{S}}}
\newcommand{\mF}{\ensuremath{\mathcal{F}}}

\newcommand{\mY}{\ensuremath{\mathcal{Y}}}
\newcommand{\mN}{\ensuremath{\mathcal{N}}}
\newcommand{\mL}{\ensuremath{\mathcal{L}}}
\newcommand{\mZ}{\ensuremath{\mathcal{Z}}}
\newcommand{\mR}{\ensuremath{\mathcal{R}}}

\newcommand{\Nm}{\ensuremath{\mathbb{N}}}
\newcommand{\Rm}{\ensuremath{\mathbb{R}}}

\newcommand{\mM}{\ensuremath{\mathcal{M}}}
\newcommand{\mP}{\ensuremath{\mathcal{P}}}
\newcommand{\mK}{\ensuremath{\mathcal{K}}}

\newcommand{\Tm}{\ensuremath{\mathbb{T}}}
\renewcommand{\sp}{\ensuremath{\mathfrak{sp}}}

\makeatletter
\newsavebox\myboxA
\newsavebox\myboxB
\newlength\mylenA
\newcommand*\xoverline[2][0.70]{%
	\sbox{\myboxA}{$\m@th#2$}%
	\setbox\myboxB\null
	\ht\myboxB=\ht\myboxA%
	\dp\myboxB=\dp\myboxA%
	\wd\myboxB=#1\wd\myboxA
	\sbox\myboxB{$\m@th\overline{\copy\myboxB}$}
	\setlength\mylenA{\the\wd\myboxA}
	\addtolength\mylenA{-\the\wd\myboxB}%
	\ifdim\wd\myboxB<\wd\myboxA%
	\rlap{\hskip 0.8\mylenA\usebox\myboxB}{\usebox\myboxA}%
	\else
	\hskip -0.5\mylenA\rlap{\usebox\myboxA}{\hskip 0.5\mylenA\usebox\myboxB}%
	\fi}
\makeatother

\begin{document}
\title[bumpy metric theorem a la Ma\~{n}e for non-convex Hamiltonians]{Bumpy metric theorem in the sense of Ma\~{n}e for non-convex Hamiltonians} 

\author[S. Aslani]{Shahriar Aslani}
\address{	University of Toronto, Department of Mathematics, 40 Saint George St., Toronto, ON, M5S 2E4.}
\email{shahriar.aslani@math.toronto.edu}

\author[P. Bernard]{Patrick Bernard}
\address{PSL Research University,
Université Paris Dauphine, CEREMADE, Place du Maréchal de Lattre de Tassigny,
75775 PARIS cedex 16, FRANCE }
\email{pbernard@ceremade.dauphine.fr} 
\begin{abstract}
We prove a bumpy metric theorem in the sense of Ma\~{n}e for  non-convex Hamiltonians that are satisfying a certain geometric property. 
\end{abstract}
\maketitle
\section{Introduction}

In the study of fiberwise convex Hamiltonian systems, Ricardo Ma\~né introduced the notion now called Ma\~né
genericity (\cite{M}): 
A property is called Mañé generic if, for each Hamiltonian $H$, the property is satisfied by the Hamiltonian $H+u$
for a generic potential $u$. 

Although this notion is particularly relevant in the case where $H$ is convex, it also makes  perfect sense for more general Hamiltonians. Our goal in the present paper is to investigate the Mañé-generic properties of periodic orbits of not necessarily convex Hamiltonian systems and to work in the direction of what could be called a bumpy metric theorem in this context: the property that all periodic orbits on a given energy surface are non-degenerate is Mañé generic. In the convex case, this problem was studied in dimension 2 by Oliveira in \cite{O08}, and the missing perturbation lemma necessary to generalize the result in any dimension was obtained in \cite{RR} (using an incorrect normal form fixed in \cite{AB}).
Observe however that there is a case not treated in \cite{O08,RR}, and that the bumpy metric theorem in the sense of Mañé 
was not solved before \cite{B24} even in the convex case, more details below. In the present paper, we generalize  the results proved in 
\cite{O08,RR} (which, as we just explained, are not exactly those stated) to non-convex Hamiltonians.

There is a long history on the study of generic properties of periodic orbits in various contexts.
It started with the works of Kupka \cite{K63} and Smale \cite{S63} stating that all periodic orbits are hyperbolic for generic systems in the class of all vector fields
(and moreover the intersections of stable and unstable manifolds are transverse). In the class of Hamiltonian systems, such a property can't be expected, but Robinson \cite{CR1, CR2}  proved among other things that generic Hamiltonian systems have no degenerate periodic orbits on a given energy surface. The difference in the present work is
 that we are considering the restricted class of perturbations by potentials.
The bumpy metric theorem, obtained in the same period, claims that generic geodesic flows have no degenerate periodic orbits, see \cite{A68,K76,A82} for contributions to this result, and \cite{BJP} for the semi-Riemannian case. Perturbing by potentials is similar (and to a large extent equivalent) to considering conformal perturbations of metrics, which is a much smaller family of perturbations than the class of all metrics. As to the study of Mañé perturbations of convex Hamiltonians, besides the works \cite{O08,RR, AB} already discussed,
let us  mention  Contreras \cite{Co10} where some important tools are introduced.

We will say that the Hamiltonian $H(q,p):T^*M  \to \mathbb{R}$ is convex if its fiberwise Hessian $\partial^2_{p^2}H(q,p)$ is positive-definite for all $(q,p) \in T^*\mathbb{R}^n$.
The following definition will play a central role in our study of Mañé generic properties of periodic orbits :

\begin{definition}\label{def: 1.6}
Let $H(q,p):T^*M \to \mathbb{R}$ be a Hamiltonian which is defined on the cotangent bundle of a smooth manifold $M.$ We say $H$ is fiberwise isoenergetically non-degenerate  at $(q,p) \in T^*M$ if 
\begin{equation} \label{eq: 1.2}
det 
\begin{bmatrix}
\partial^2_{pp}H(q,q) & \partial_p H(q,p) \\ 
\partial_p H(q,p)^T & 0 
\end{bmatrix} \ne 0. 
\end{equation}
\end{definition}

This is equivalent to saying that $p$ is a regular point of the function $H(q,.)$ on $T^*_qM$, and that the Hessian of this function, seen as a quadratic form, is non-degenerate on the kernel 
of its differential. This kernel is the intersection between the tangent space of the energy level and the fiber.

A convex Hamiltonian is fiberwise isoenergetically non-degenerate at each point except those where $\partial_pH=0$. There is
at most  one such point per fiber.

Another important example to have in mind is the case of fiberwise quadratic Hamiltonians. If such a Hamiltonian is non-degenerate in each fiber, then it is fiberwise isoenergetically non-degenerate precisely outside of its zero energy level. 

Given a periodic orbit $\theta$ of the Hamiltonian  $H$, we can as is usual take a transverse section and consider the Poincaré return map to that section. This 
Poincaré map preserves the energy level, and we call restricted Poincaré map its restriction  to the energy level. 
The differential at the orbit of this restricted map is called the restricted linearized return map. It is well defined and symplectic on the tangent space to the restricted section. By taking a symplectic base of this tangent space, we can consider the restricted linearized return map as an element of the symplectic group
$Sp(2d)$ (if $M$ has dimension $d+1$). Up to conjugacy, it does not depend on the section or the base. 
The orbit is called non-degenerate if $1$ is not an eigenvalue of the restricted linearized map.

Given a periodic orbit $\theta(t)=(Q(t),P(t))$, of minimal period $T$, we say that $s$ is a neat time  if  $\dot Q(s)\neq0$ and if  there exists no $t\neq s \mod T$
 such that $Q(s)=Q(t)$. 
With this definition,  the set of neat  times  is easily seen to be open  (see the proof of Lemma \ref{lem-open} below).
Note, in contrast to what is implicitly assumed in \cite{O08, RR}, that periodic orbits without any neat time may exist, even in the convex case. For example, if $H$ is a natural system of the form $H(q,p)=g_q(p,p)+u(q)$, where $g$ is a Riemannian metric and $u$ is a potential defined on the base, then  there often exist reversible periodic orbits, \textit{i.e.} periodic orbits which perform a round trip above an arc in the base. The existence of such orbits is studied for example in \cite{K76}, where they are called librations. They have no neat time.

Given an orbit $\theta$ of the Hamiltonian $H$, we say that the potential $u$ is admissible for $\theta$ if the value and the differential of $u$ are both zero at each point of the projection of $\theta$. We denote by  $C^{\infty}_{\theta}(M)\subset C^{\infty}(M)$ the space of admissible potentials.
If $u$ is admissible for the periodic orbit $\theta$, then $\theta$ is also a periodic orbit of $H+u$.
 Given a transverse section to $\theta$, the restricted transverse section of $\theta$ for $H+u$ is different from the restricted transverse section of $\theta$ for $H$, but they have the same tangent space at the orbit. So by choosing a symplectic base of this fixed tangent space, we can define the restricted linearized return map 
 $L(\theta,H+u), u\in C^{\infty}_{\theta}(M)$ as an element of $Sp(2d)$. The map $u\mapsto L(\theta,H+u)$ is well defined up to a fixed conjugacy.
Our first result is:

\begin{theorem}\label{thm-pert}
Consider a smooth Hamiltonian $H(q,p):T^*M\to \mathbb{R}$, and a periodic orbit $\theta(t)$ of the Hamiltonian vector field of $H$. Assume that $\theta$ admits  a neat time $t_0 \in \mathbb{R}$  such that $H$ is fiberwise isoenergetically non-degenerate at $\theta(t_0)$. Then   the map
$$
C^{\infty}_{\theta}(M)\ni u \longmapsto L(\theta, H+u)\in Sp(2d)
$$  
is weakly open, meaning that the image of each non-empty open set contains a non-empty open set.
\end{theorem} 

This theorem is proved below in Section \ref{sec-pert} using a local normal form stated and proved in Section \ref{sec-nf}.
It is likely that the above map is actually open,  an  adaptation of \cite{LRR} should provide a proof.
In the convex case, this perturbation statement was obtained in \cite{RR}, 
using a normal form from \cite{FR2} which  was corrected in \cite{AB}.
Our proof  follows a similar strategy.
 In the convex case, it is automatic that $H$ is fiberwise isoenergetically non-degenerate at $\theta(t_0)$ when $t_0$ is a neat time, so this assumption can be omitted.
However, the assumption that $\theta$ admits a neat time is 
necessary, and not automatic. It is wrongly omitted in \cite{O08,RR}. 
To see that the existence of a neat time is a necessary assumption, consider again a natural system   $H=g_q(p,p)+v(q)$ and an orbit $\theta$ which is a libration (a reversible orbit).
Then $H+u$ is still reversible for each potential $u$, meaning that $(H+u)(q,-p)=(H+u)(q,p)$ and as a consequence 
the section and the coordinates can be chosen such that $L(\theta,H+u)$ is a reversible symplectic matrix for each $u\in C^{\infty}_{\theta}(M)$, meaning that $R=LRL$,
where 
$R=\begin{bmatrix}
I&0\\0&-I
\end{bmatrix}$.
Since the space of symplectic reversible matrices is a submanifold of positive codimension in $Sp(2d)$,
 the image of the map $u\mapsto L(\theta, H+u)$ has  no interior, which contradicts the conclusion of the
theorem in that case.
\footnote{See \cite{B24}, a paper written after this one was submitted, for more details.}

It is also clear that some nondegeneracy assumption on $H$ is necessary.
Indeed, we can consider the Hamiltonian $H=p_1$ on the manifold $M=\Tm\times \Rm^d$ with coordinates $(q_1,\hat q)$.
All the Hamiltonians $H+u$ generate the equations $q_1'=1, \hat q'=0$ so, in the sections
$q_1$ constant, the return map of any periodic orbit is fixing the $\hat q$ coordinate, hence the linearized return maps  have a first block line equal to $[I,0]$. Such matrices have no interior in 
$Sp(2d)$.

By methods similar to those used in \cite{A63, A68,A82,O08}, Theorem \ref{thm-pert} implies :

\begin{theorem}\label{thm-g1}
	Given a smooth Hamiltonian $H: T^*M\rightarrow \Rm$, and a conjugacy invariant subset $\Upsilon \subset Sp(2d)$
	which is an $F_{\sigma}$ with empty interior,
	there exists an $F_{\sigma}$ with empty interior  $\mF\subset C^{\infty}(M)$   such that, for each $u\in C^{\infty}(M)-\mF$, the Hamiltonian system $H+u$ has the following property :
	
	The zero energy level of $H+u$ is regular. Moreover,  
	if  $\theta$ is a zero energy periodic orbit of $H+u$ 
	which admits a neat time $t_0$ such that $H$ is fiberwise isoenergetically non-degenerate at $\theta(t_0)$, then $\theta$ is non-degenerate and  satisfies 
	$$
	L(\theta, H+u)\not \in \Upsilon.
	$$
\end{theorem}

Recall that a set is called an $F_{\sigma}$ if it is a countable union of closed sets.
Theorem \ref{thm-g1} is proved in section \ref{sec-pt}. The main lines in the proof are similar to those used in \cite{A82,O08},
but our presentation is different and avoids the recurrence on the periods.

The conclusion that $\theta$ is non-degenerate is actually contained in the conclusion 
$
L(\theta, H+u)\not \in \Upsilon,
$
provided $\Upsilon$ contains matrices having an eigenvalue equal to one, which will be assumed without loss of generality.

As mentioned earlier, the convex case  was  obtained in dimension 2 ($d=1$) by Oliveira in \cite{O08}, and then in any dimension 
by Rifford and Ruggiero in \cite{RR} (complemented by \cite{AB}). Our statement is weaker than those in these papers, where the conclusions are claimed for all periodic orbits.
However, the existence of a neat time $t_0$ is  implicitly used in these papers, and it is wrongly omitted in the statements. In the convex case, once $t_0$ is a neat time, it is automatic that $H$ is fiberwise isoenergetically non-degenerate at $\theta(t_0)$.

In the non-convex  situation studied here, we have the second undesirable constraint  that $H$ be fiberwise isoenergetically non-degenerate at $\theta(t_0)$. 
This can obviously be ensured by assuming that $H$ is fiberwise isoenergetically non-degenerate at each point $x$ of $T^*M$
where $\partial_pH(x)\neq 0$.
This however is a  strong assumption which is satisfied  in the convex case, but not when $H$ is a  non-convex fiberwise quadratic Hamiltonian.
An important observation here is that the set $\Sigma_H$ of points of $T^*M$ at which $H$ fails to be fiberwise isoenergetically non-degenerate is a fixed data of the problem, which is unchanged by adding a potential. 
We now give a hypothesis on this set $\Sigma_H$ which is sufficient to generically ensure the condition that $H$ is fiberwise isoenergetically non-degenerate at $\theta(t_0)$. 
This condition is satisfied by the set $\Sigma_H$ associated to  non-degenerate fiberwise quadratic Hamiltonians.

\begin{hyp}\label{h-1}
	The subset $\Sigma\subset T^*M$ is contained in a countable union of manifolds of positive codimension which are transversal to the vertical.
\end{hyp}

\begin{theorem}\label{thm-g2}
Let $H: T^*M\rightarrow \Rm$ be a smooth Hamiltonian and let $\Sigma$ be a subset of $T^*M$ satisfying Hypothesis \ref{h-1}. 
There exists an $F_{\sigma}$ with empty interior  $\mF\subset C^{\infty}(M)$  such that, for each $u\in C^{\infty}(M)-\mF$, the Hamiltonian system $H+u$ has the following property :

For each orbit  $\theta$ of $H+u$, and each time $t_0$ such that $\partial_pH(\theta(t_0))\neq 0$,  there exist  times $t$, arbitrarily close to $t_0$, such that 
$\theta(t)\not \in \Sigma$.
\end{theorem}

This theorem is proved in Section \ref{sec-o}.
It can be applied in particular to the case where $\Sigma$ is the subset $\Sigma_H$ of points at which $H$ is not fiberwise isoenergetically non-degenerate, and we obtain:

\begin{theorem}\label{g3}
	Let   $H: T^*M\rightarrow \Rm$, be a Hamiltonian such that $\Sigma_H$ satisfies Hypothesis \ref{h-1} and let $\Upsilon \subset Sp(2d)$ be a conjugacy invariant $F_{\sigma}$ with empty interior.  
There exists an $F_{\sigma}$ with empty interior   $\mF\subset C^{\infty}(M)$  such that, for each $u\in C^{\infty}(M)-\mF$, the Hamiltonian system $H+u$ has the following property :

	The zero energy level of $H+u$ is regular. Moreover,  
	if  $\theta$ is a zero energy periodic orbit of $H+u$ 
	which admits a neat time $t_0$, then $\theta$ is non-degenerate and  satisfies 
	$$
	L(\theta, H+u)\not \in \Upsilon.
	$$
\end{theorem}

\textit{Notations.} 
We denote by $d+1$ the dimension of $M$, so that the symplectic sections have dimension $2d$.
We denote by $\mM(m), \mS(m),\mS^-(m)$ respectively the spaces of square matrices of size $m$, of symmetric matrices of size $m$,
and of antisymmetric matrices of size $m$. Finally, we denote by $Sp(m)$  ($m$ even) the group of symplectic matrices and  by
by $\sp(m)$ the space of Hamiltonian matrices, which are the matrices $L$ such that $\mathbb{J}L$ is symmetric, where  
$\mathbb{J}$ is the standard symplectic matrix $\begin{bmatrix}
0 & I \\
-I & 0
\end{bmatrix}$. 
We will usually denote by $q=(q_1, \ldots, q_{d+1})=(q_1, \hat q)$ the local coordinates on $M$ and by 
$x=(q,p)= (x_1, \hat x)$, with $x_1=(q_1,p_1)$ and $\hat x=(\hat q, \hat p)$, the local coordinates on $T^*M$.
We denote by $(e_i)$ the canonical bases of $\Rm^m$ and $\Rm^{m*}$.

We will denote by $\varphi(t,x,u)$ or $\varphi^t_u(x)$ the Hamiltonian flow of $H+u$.
Note that the space $C^{\infty}(M)$ of potentials $u$ is a separable Fréchet space, but not a Banach space.
Although names may suggest the opposite, there is no simple notion of Fréchet differential on a Fréchet space.
The notation $\partial_u$ of partial derivative with respect to the variable $u\in C^{\infty}(M)$ will always be understood in the meaning of Gateau differentiability. We will restrict to finite dimensional subspaces of  $C^{\infty}(M)$ whenever we really
manipulate differential calculus (and then refer to the stronger notion of Fréchet differential).

After the present paper was submitted, some of the problems presented as open in this introduction were solved in the convex case in  \cite{B24}. Although finally published later, the present paper was submitted before \cite{B24} was written. 

\subsection{Acknowledgement} 
For the purpose of Open Access, a CC-BY-SA public copyright licence has been applied by the authors to the present document and will be applied to all subsequent versions up to the Author Accepted Manuscript arising from this submission.

\section{Normal form near orbit segments of non-convex Hamiltonian systems}\label{sec-nf}

In this section, we give a normal form near fiberwise isoenergetically non-degenerate points which will be used to derive 
Theorem \ref{thm-pert}. This normal form and its proof are similar to the one obtained in \cite{AB} in the convex setting,
and the  novelty here is  to single out fiberwise isoenergetic non-degeneracy as the appropriate hypothesis.
Since potentials, or in other words Hamiltonians constant on the fibers, play a special role in the problem, 
we need to restrict to symplectic charts which preserve the vertical fibration, 
and more precisely charts $\Psi: U\times (\Rm^{d+1})^*\rightarrow  T^*M$ which send each fiber $q\times  (\Rm^{d+1})^*$
to a fiber of $T^*M$ (not necessarily in a linear way).

\begin{theorem}\label{thm-nf}
	Let $H:T^*M\to \mathbb{R}$ be a smooth Hamiltonian, and let $x_0\in T^*M$ be a point such that $H$ is fiberwise isoenergetically non-degenerate at $x_0$. Then, there exists a fibered symplectic chart near $x_0$ 
	and $\delta >0$ such that, in these coordinates,  for $|t|\leq \delta$, we have:
	\begin{itemize}
		\item [(1)] $\varphi^t_H(0)=(te_1,0)$
		\item [(2)] $\partial^2_{q \hat p}H(t e_1,0)=0.$
		\item [(3)]  $\partial^2_{p_1\hat{p}}H(t e_1,0)=0.$ 
		\item [(4)] $\partial^2_{\hat{p}\hat p}H(t e_1,0)=D,$  
	\end{itemize}
where $D$ is a diagonal matrix whose diagonal elements are  $\pm 1$, and where we denote by $\varphi^t_H$ the Hamiltonian flow of $H$.
\end{theorem}
\begin{remark}
	The equalities (1) to (4) are equivalent to the expansion
$$
H(q,p)= H(q,0)+ v(q)p_1 + \frac{1}{2}a(q_1)p_1^2+ \frac{1}{2}\hat p D \hat p^t +O_3(\hat q, p)
$$
with $H(q,0)= H(0,0)+O_2(\hat q)$ and $v(q) =1+ O(\hat q)$, for $|q_1|\leq \delta$.
Indeed, we have the second order expansion in $p$
\begin{align*}
H(q,p) &= H(q,0)+ p\cdot V(q)  +\frac {1}{2} p\partial^2 _{pp}H(q,0) p^t +O_3(p)\\
       &= H(q,0)+ p\cdot V(q)  +\frac {1}{2}p \partial^2 _{pp}H(q_1e_1,0) p^t+O_3(\hat q,p)
\end{align*}
with
$
V(q) := \partial_pH(q,0)
$.
If (1) holds, we have $V(q_1,0)=e_1$, so
$$v(q):= \partial_{p_1} H(q,0)= 1+O(\hat q)
$$
and assuming moreover (2) yields
$
\partial_{\hat p} H(q,0)=O_2(\hat q)
$
hence
$$
p\cdot V(q) = v(q)p_1+O_3(\hat q, p).
$$
The quadratic terms in $p$ then take the desired form with $a(q_1)= \partial^2_{p_1p_1}H(q_1e_1,0)$  assuming (3) and (4).

\end{remark}

\textit{Proof of Theorem \ref{thm-nf}.}  
By using a chart on $M$ at $q_0=\pi(x_0)$, we can assume that $M=\Rm^{d+1}$.
Since $H$ is fiberwise isoenergetically non-degenerate at $x_0$, we have $\partial_pH(x_0) \neq 0$, which implies that the projected orbit 
$t\mapsto \pi \circ \varphi^t_H(x_0)$ is an embedding near $t=0$.
This allows to take the chart in such a way that, for small $t$, 
\begin{equation}\label{eqe1}
\pi \circ \varphi^t_H(x_0)=te_1.
\end{equation}
We denote by $\underline H(q,p)$ the Hamiltonian in these local  coordinates.
Since $\partial_p\underline H(0)=e_1$ the hypothesis that $H$ is fiberwise isoenergetically non-degenerate at $x_0$ is equivalent in coordinates to the 
fact that 
$
\partial^2_{\hat p\hat p}\underline H(0)
$
is invertible.

\begin{definition}
(a) We call a symplectic map $\Psi(q,p):T^*\mathbb{R}^{d+1} \to T^*\mathbb{R}^{d+1}$ \textit{fibered} if it preserves the vertical fibration \textit{i.e.} if it has the form $\Psi(q,p)=\big(\phi(q),G(q,p)\big)$. \\ 
(b) We say that the fibered symplectic map $\Psi$ is \textit{homogeneous} if it preserves the zero section, then its is of the form $\Psi(q,p)=\big(\phi(q), p\circ (d\phi_q)^{-1}\big)$ for some local diffeomorphism $\phi$ of the base. \\
(c) We say that $\Psi$ is \textit{vertical} if it preserves each fiber, then it is locally of the form $\Psi(q,p)=\big(q, p + d g(q)\big)$ for some function $g(q):\mathbb{R}^{d+1} \to \mathbb{R}$. \\
(d) We say that $\Psi$ is \textit{admissible} if it is a fibered symplectic local diffeomorphism, and if its horizontal
component is the identity on a small interval $te_1, |t|\leq \delta$.
\end{definition}

Each fibered symplectic diffeomorphism is the composition of a homogeneous and of a vertical one.

Let us  detail the elementary proofs on the above claims.
An easy and classical computation shows that a lower  block triangular matrix $M$ is symplectic (meaning that $M^T \mathbb{J} M=\mathbb{J}$)  if and only if it is of the form
$$M=\begin{bmatrix} A &0 \\ BA& (A^{-1})^T\end{bmatrix}$$
with $B$ symmetric. So the second coordinate $G(q,p)$ of a fibered symplectic diffeomorphism has the form $G(q,p)=\omega(q)+p\circ (d\phi_q)^{-1} $
(recalling that $A^Tp=p\circ A$). If the diffeomorphism preserves the zero section, then $\omega\equiv 0$ and 
$\Psi$ has the form described in (b) above.

Note that  $\omega(q)=G(q,0)$, and $\omega$ is a map from $\Rm^{d+1}$ to $(\Rm^{d+1})^*$, in other words it is a one-form on $\Rm^{d+1}$. It is well-known 
that there exists a function $g$ on $\Rm^{d+1}$ such that $\omega=dg$ if and only if $\omega$ is a closed one-form, meaning in coordinates that 
$\partial_j\omega_i=\partial_i\omega_j$ for each indices $i$ and $j$.

We also deduce from the expression $G(q,p)=\omega(q)+p\circ (d\phi_q)^{-1} $ that a vertical symplectic diffeomorphism has the form 
$\Psi(q,p)=(q, p+\omega(q)) $. Its differential in coordinates is 
$\begin{bmatrix} I&0\\ \partial_q\omega &I\end{bmatrix}$, and it is a symplectic matrix. From the above remark on block triangular symplectic matrices, we deduce that $\partial_q\omega$ is symmetric at each point, which is equivalent to $\omega$ being closed. This proves the claim in (c).

Finally, if $\Psi(q,p)= (\phi(q), \omega(q)+ p\circ (d\phi_q)^{-1})$ is a general fibered symplectic diffeomorphism, and 
$\Psi_0(q,p):=(\phi(q),  p\circ (d\phi_q)^{-1})$ is the corresponding homogeneous diffeomorphism (which is still symplectic), we  observe that
$\Psi\circ \Psi_0^{-1}$ and $\Psi_0^{-1}\circ \Psi$ are vertical symplectic diffeomorphisms, so $\Psi$ is the composition of a homogeneous symplectic diffeomorphism and of a vertical one (in any order).  \qed

We shall now prove the conclusions of the Theorem by applying a succession of fibered symplectic diffeomorphisms fixing the points $(q_1e_1,0)$, hence preserving (1). Such diffeomorphisms will be called admissible.
In each step, we denote by $\underline H$ the original Hamiltonian, and by $H=\underline H\circ \Psi$ the Hamiltonian in the new coordinates. We allow ourselves to reduce $\delta$ at each step. 

\begin{proof}[Proof of (1)]
Let $\underline P_1(t)$ be the first component of $\underline P(t)$ (the vertical component of the orbit before the change of coordinates).
We consider a function $v(t) : \Rm\rightarrow  \Rm$ such that  $v'=\underline P_1$
and the function 
$u(q_1, \hat q):=v(q_1)+\underline{\hat P}(q_1)\cdot \hat q.
$
We have $du _{te_1}=\underline P(t)$, hence applying the vertical diffeomorphism $\Psi (q,p)=(q, p+du_q)$,
the new orbit $(Q(t), P(t))=\Psi^{-1}(\underline Q(t), \underline P(t))$ satisfies $P(t)=0$ for small $t$. 
\end{proof}

\begin{proof}[Proof of (2)]
	We assume that (1) holds.
	We consider the vector field
	$$
	\underline V(q):= \partial _p \underline H(q, 0)
	$$
	on $\Rm^{d+1}$.
	We apply the flow box Theorem to find  a local diffeomorphism $\phi$ of $\Rm^{d+1}$ near $0$ 
	such that the backward image of $\underline V$ by $\phi$ is the constant vector field 
	$V(q)\equiv e_1$. 
	Since the initial orbit of $0$ for the flow of $\underline V$ is already $t\mapsto te_1$, we can moreover assume that 
	$\phi$ is fixing the points $te_1$, $|t|\leq \delta$.
	
	Let $\Psi(q,p)=(\phi(q), p\circ (d\phi_q)^{-1})$ be the corresponding homogeneous diffeomorphism and 
	$H=\underline H \circ \Psi$. Since $\Psi$ is symplectic, it sends the Hamiltonian vectorfield $X_H$ of $H$ to 
	the Hamiltonian vectorfield $X_{\underline H}$ of $\underline H$, meaning that 
	$$
	X_{\underline H}(\Psi(q,p))=d\Psi_{(q,p)}\cdot X_H(q,p),
	$$
	and in particular
	$$
	X_{\underline H}(\phi(q),0)=d\Psi_{(q,0)}\cdot X_H(q,0).
	$$
	Looking at the horizontal components, we get
	$$
	\underline V(\phi(q))= d\phi_q\cdot \partial_pH(q,0)
	$$
	which means that $V(q):=\partial_pH(q,0)$ is the backward image of $\underline V$ under $\phi$.
	The way we chose $\phi$ implies that 
	$$\partial_pH(q,0)=V(q)\equiv e_1$$ 
	locally.
	This is quite stronger than (2), but we will apply other changes of coordinates which will not preserve this additional 
	structure.
	\end{proof}

	The two steps just performed do not use the hypothesis of fiberwise isoenergetic non-degeneracy, but this hypothesis is essential for the next steps. In coordinates such that (1) holds, this hypothesis is equivalent to the invertibility of the matrix
	$
	\partial^2_{\hat p\hat p}H(0).
	$
	
\begin{proof}[Proof of (3).]	
We assume that (1) and (2) are  satisfied for $\underline H$, and prove that 
(3) can be obtained by a further admissible change of coordinates.
We consider a base diffeomorphism of the form
$\phi (q_1,\hat q)= (q_1+l(q_1)\cdot \hat q, \hat q)$,
where $q_1\mapsto l(q_1)$ is a smooth map with values in $\Rm^{d*}$ defined near $0$ in $\Rm$.
The corresponding homogeneous diffeomorphism satisfies
$$
\Psi: (q_1,0,p_1,\hat p)\mapsto(q_1, 0, p_1, \hat p-p_1 l(q_1)).
$$
We then have 
$\partial _{\hat p}(\underline H\circ \Psi)_{(q_1e_1, p_1, \hat p)}=
\partial _{\hat p}\underline H _{(q_1e_1, p_1, \hat p-p_1l(q_1))}
$
hence
$$
\partial^2_{p_1 \hat p}H_{(q_1e_1,0)}=\partial^2_{p_1 \hat p}(\underline H\circ \Psi)_{(q_1e_1,0)}=\partial^2 _{p_1\hat p}\underline H _{(q_1e_1,0)}-
\partial^2_{\hat p\hat p}\underline H_{(q_1e_1,0)}\cdot l^t(q_1).
$$ 
We obtain (3) by choosing 
$$
l(q_1):= \big( (\partial^2_{\hat p\hat p}\underline H_{(q_1e_1,0)})^{-1} \cdot
\partial^2 _{p_1\hat p}\underline H _{(q_1e_1,0)}\big)^t.
$$
Observe that  $\partial^2_{\hat p\hat p}\underline H_{(q_1e_1,0)}$ is invertible for small $q_1$ as a consequence of the hypothesis
of fiberwise isoenergetic non-degeneracy. 
Note in the above computation that we consider $l$ as a line matrix.

Finally, let us check that we have preserved (2). As in  the proof of (2), the diffeomorphism $\phi$ sends the vectorfield 
$V(q)=\partial_p H(q,0)$ to the vectorfield $\underline V(q)=\partial_p\underline H(q,0)$. 
Since 
 $\phi$ preserves the $\hat q$ 
coordinate, the corresponding coordinates of $V$ and $\underline V$ satisfy
$$
\hat V = \underline {\hat V} \circ \phi.
$$
Differentiating this equality at the point $q_1e_1$ gives 
$$
\partial_q \hat V(q_1e_1)
=\partial _q \underline {\hat V} (q_1e_1)\cdot \partial_q \phi(q_1e_1).
$$
If $\underline H$ satisfies (2), then $\partial_q\underline {\hat V} (q_1e_1)=0$, so
$$
\partial_{q\hat p} H(q_1e_1, 0)=\partial  _q {\hat V} (q_1e_1)=0.
$$
\end{proof}

\begin{proof}[Proof of (4).]
	
	We assume that the equations (1) to (3) initially hold.
	We will obtain (4) by an admissible (usually not homogeneous) transformation preserving all these equalities. This transformation will be decomposed  into first a homogeneous tranformation and second a vertical
	transformation none of which  preserve (2).

	The first step consists of applying the homogeneous change of coordinates $\Psi$ associated to a diffeomorphism
	of the form
	$$
	\phi (q_1,\hat q)=(q_1, M(q_1)\cdot \hat q),
	$$
	where $M(t)$ is a $d\times d$  invertible matrix depending smoothly on $t\in \Rm$ near $t=0$.
	The matrix of the differential of $\phi$ is 
	$$
	J (q)=
	\begin{bmatrix}
	1 & 0\\ M'(q_1)\hat q & M(q_1)
	\end{bmatrix},
	\quad 
	J^{-1}(q)=
	\begin{bmatrix}
	1 & 0\\- M^{-1}(q_1)M'(q_1) \hat q & M^{-1}(q_1)
	\end{bmatrix},
	$$
	where $M'(q_1)$ is the derivative.
	We thus have
	$$
	\Psi(q,p)=\big(q_1, M(q_1) \hat q, p_1-\hat p M^{-1}(q_1)M'(q_1)  \hat q 
	,\hat p M^{-1}(q_1)\big).
	$$
	The Hamiltonian in original coordinates is of the form
	$$
	\underline H (q,p)=
	\underline H(q,0)+ \underline v(q) p_1+ \frac{1}{2} \underline a(q_1) p_1^2 + \frac{1}{2}   \hat p\underline A(q_1) \hat p^t +O_3(\hat q, p)
	$$
	for small $|q_1|$,
	with $\underline v(q)=\partial_{p_1}\underline H(q,0)=1+O(\hat q)$, $\underline a(q_1)=\partial^2_{p_1p_1} \underline H(q_1e_1, 0)$ and
	$\underline A(q_1)=\partial^2_{\hat p\hat p} \underline H(q_1e_1,0)$.
	We compute
	\begin{align*}
	H(q,p)&=\underline H\circ \Psi (q,p) =
 H(q,0)+ v(q)\big(p_1-\hat p M^{-1}(q_1)M'(q_1) \hat q\big) \\ &+\frac{1}{2}  a(q_1) (p_1+O_2 (\hat q, \hat p))^2 
	 +\frac{1}{2}   \hat pM^{-1}(q_1)\underline A(q_1) (\hat pM^{-1}(q_1))^t + O_3(\hat q, p)
	\end{align*}
	where $v(q)=\underline v\circ \phi(q)=1+O(\hat q)$ and $a(q_1)=\underline a(q_1)$. We obtain
\begin{align*}
	H(q,p) &=
	H(q,0)+ v(q)p_1-\hat p M^{-1}(q_1)M'(q_1) \hat q+\frac{1}{2}  a(q_1) p_1^2 \\
	& +\frac{1}{2}   \hat pM^{-1}(q_1)\underline A(q_1)(M^{-1}(q_1))^t\hat p^t+ O_3(\hat q, p).
\end{align*}
Since $\underline A(q_1)$ is assumed invertible for $q_1=0$, there exists a diagonal matrix $D$ with diagonal terms 
	equal to $\pm 1$, and a matrix $M(0)$ such that $M(0)DM^t (0)=\underline A(0)$, which implies that 
	$\partial^2_{\hat p\hat p}H(0,0)=D$.
	Moreover, there exists a smooth curve $M(t)$ of matrices, defined near $t=0$ such that 
	$M(t)DM^t(t)=\underline A(t)$ for all small $t$ (we shall actually construct such a curve $M(t)$ below)
	and this implies that $A(t):= \partial^2_{\hat p\hat p}H(te_1,0)$ is constant and equal to $D$ for small $t$.

	However, the unavoidable  apparition of the term $\hat pM^{-1}(q_0)M'(q_0) \hat q$ means that (2) has not necessarily been preserved. In order to be able to restore it by a vertical change of coordinates, we need a particular choice for the curve  $M(t)$ :
	
	\begin{lemma}
		We can choose $M(t)$ in such a way that 
		$$
		H(q,p)=
		H(q,0)+v(q)p_1-\hat pDB(q_1) \hat q+\frac{1}{2}  a(q_1) p_1^2 
		+\frac{1}{2}  \hat pD\hat p^t + O_3(\hat q, p),
		$$ 
		where $B(q_1)$ is symmetric for all small $q_1$.
	\end{lemma} 
	
	\proof
	We need the matrix $M(t)$ to satisfy the two conditions that $M(t) DM^t(t)=\underline A(t)$ and $B(t):=D M^{-1}(t)M'(t)$ is 
	symmetric (recall that $D^2=Id$).
	Derivating the first condition, we get $M'DM^t +MD(M')^t =\underline A'$.
	Assuming the symmetry of $B$, we have that $(M^t)'(M^t)^{-1}D=DM^{-1}M'$, which implies 
	$M D(M^t)'=M'DM^t$. We obtain the equation $2M'DM^t=\underline A'$ or in other words :
	$$
	M'(t)=\underline A'(t)(M^t)^{-1}D/2.
	$$
	Reducing $\delta$ if necessary, there exists a solution $M(t)$ of this differential equation on the interval $[-\delta, \delta]$, with an initial condition $M(0)$ satisfying 
	$M(0)DM^t(0)=\underline A(0)$ (such an $M(0)$ exists provided the number of $-1$ on the diagonal on $D$ is equal to the signature of $A(0)$).
	For such a solution, we see that the corresponding $B$ is 
	$$
	DM^{-1}M'=DM^{-1}\underline A' (M^t)^{-1}D/2
	$$
	and so it is symmetric.
	Then, the  computation made earlier shows that 
	$(MDM^t)'=\underline A'$. Since the equality $MDM^t=\underline A$ is satisfied at $t=0$, it is thus satisfied for all $t\in [-\delta, \delta]$.
	\qed
	
	The second step consists of applying the vertical change of coordinates 
	$$
	\Theta :(q,p)\mapsto (q,p+du_q)
	$$
	with $u(q)=  \hat q^tB(q_1)\hat q/2$, so that $du_q=(\alpha(q),\hat q^tB(q_1))$, with $\alpha(q)= O_2(\hat q)$. We  obtain :  
	\begin{align*}
	H\circ \Theta (q,p)=&
	H(q,0)+v(q)(p_1+\alpha(q))-\big(\hat p+\hat q^tB(q_1)\big)DB(q_1) \hat q+\frac{1}{2}  a(q_1) (p_1+\alpha(q))^2 \\
	&+\frac{1}{2} \big(\hat p+\hat q^tB(q_1)\big)D\big( \hat p^t+B(q_1)\hat q\big) + O_3(\hat q, p)\\
	=&H(q,0)+v(q)\alpha(q)-\frac{1}{2} \hat q^tB(q_1)DB(q_1)\hat q
	+\frac{1}{2}a(q_1) \alpha^2(q)\\
	&+	v(q)p_1+a(q_1)p_1\alpha(q)+\frac{1}{2}  a(q_1) p_1^2 \\
	&-\hat pDB(q_1) \hat q+\frac{1}{2} \hat pDB(q_1)\hat q  + \frac{1}{2}\hat q^tB(q_1)D\hat p^t 
	+\frac{1}{2}\hat pD\hat p^t+ O_3(\hat q, p)\\
	=&f(q)
	+w(q)p_1+\frac{1}{2}  a(q_1) p_1^2 
	+\frac{1}{2} \hat pD\hat p^t+ O_3(\hat q, p),
	\end{align*}
	for some $f$ and $w$ satisfying  $f(q)=H(0,0)+O_2(\hat q)$ and  $w(q)=1+O(\hat q)$.
	Note in the above computation that $\hat q^tB(q_1)D\hat p^t =\hat pDB(q_1)\hat q$ because this is a $1\times 1$, hence  symmetric, matrix.
	\end{proof}

\section{Perturbing the linearized maps.} \label{sec-pert}

We prove  Theorem \ref{thm-pert}.
We consider a periodic orbit $\theta$ of $H$, of minimal period $T$.
We assume that there exists a neat time $t_0$ such that $H$ is fiberwise isoenergetically non-degenerate at $\theta(t_0)$, and there is no loss of generality in assuming that $t_0=0$.
We work locally near $x_0=\theta(0)$, in the coordinates given by Theorem \ref{thm-nf}. In these coordinates, $x_0=\theta(0)=(0,0)$.

The Hamiltonian flow of $H$  defines  Poincaré transition maps along the orbit $(te_1,0)$
between the sections $\{q_1=0\}$ and $\{q_1=t\}$ for $|t|\leq \delta$.
We shall be mostly interested in the restriction of this transition map to the energy level $\{H=0\}$, called the restricted transition map.
In the local coordinates, we have $dH(te_1,0)=(0,e_1)$, hence the tangent space to the energy level along the orbit is $\{p_1=0\}$, and the tangent space of the restricted section 
$\{H=0\}\cap \{q_1=t\}$ is  the space $\{q_1=0, p_1=0\}$, which we identify symplectically  with $\mathbb{R}^{2d}$
with coordinates $\hat x=(\hat q, \hat p)$.
The differential at $0$ of the restricted transition map between the sections $\{q_1=0\}$ and $\{q_1=t\}$ is then a symplectic linear map
$
L(t)\in Sp(2d).
$

\begin{lemma}\label{lem-el}
Assume that $H:T^*\Rm^{d+1} \to \mathbb{R}$ satisfies the conclusions  of Theorem \ref{thm-nf}, then the  restricted linearized transition maps $L(t)$ solve the differential equation
 $$\dot{L}(t)=Y(t)L(t),
 $$
 where 
\begin{align*}
 Y(t)&:=\mathbb{J}\partial^2_{\hat{x}^2}H(t e_1,0)=
 \begin{bmatrix}
 0 & D \\
 -K(t) & 0
 \end{bmatrix},\\
  K(t)&:=\partial^2_{\hat{q}^2}H(t e_1,0), \quad D:=\partial^2_{\hat{p}^2}H(t e_1,0).
\end{align*}

\end{lemma}

\begin{proof}
	Let $X_H$ be the Hamiltonian vectorfield of $H$, and let $Z:=X_H/\partial_{p_1} H$.
	In other words, $Z$ is the reparametrisation of $X_H$ which has a first coordinate equal to $1$. 
	The flow of $Z$ defines the same local transition   maps  as $X_H$ between the sections $\{q_1=0\}$ and $\{q_1=t\}$, 
	and this map is $(\hat q, p)\mapsto \varphi^t_Z(0,\hat q, p)$ (the definition of $Z$ ensures that this point does belong to the section $\{q_1=t\}$).
	 Denoting $f=1/\partial_{p_1}H$, the expression  $dZ(te_1,0)=fdX_H(te_1,0)+ (e_1,0)df(te_1,0)$ shows that only the $q_1$ equation in the linearized system of $Z$ is different from the one of $X_H$.

The other equations of the linearized system associated to $X_H$ 
	 along the orbit $(te_1,0)$ are 
	 \begin{align*}
	 &
	 p'_1(t)=0, \\
	 &\hat x'(t)=Y(t)\hat x(t)+\mathbb{J}(\partial^2 H/\partial q_1\partial \hat x)q_1+\mathbb{J}(\partial^2 H/\partial {p_1}\partial \hat x)p_1.
	 \end{align*}
	 The equation $p_1'=0$ is the linearized counterpart of the preservation of the energy.
	 As to the second equation, it follows from a differentiation of the equation of motion
	 $$
	 \hat x'(t)=\mathbb{J} \partial_{\hat x} H(q_1(t), p_1(t),\hat x(t))
	 $$
	 along the orbit $(q_1(t), p_1(t), \hat x(t))=(t,0,0)$.
	 So the linearized system of 
	  $Z$ along the orbit $(te_1,0)$ 
	is given by the equations 
	\begin{align*}
	&q'_1(t)=0\quad,\quad
	p'_1(t)=0, \\
	&\hat x'(t)=Y(t)\hat x(t)+\mathbb{J}(\partial^2 H/\partial q_1\partial \hat x)q_1+\mathbb{J}(\partial^2 H/\partial {p_1}\partial \hat x)p_1.
	\end{align*}
This non-autonomous system of linear equations preserves the subspace $\{q_1=0, p_1=0\}$ and on this subspace the evolution is governed by the subsystem $\hat x'(t)= Y(t) \hat x(t)$.
\end{proof}

We are now interested in the linearized transition map of the Hamiltonian $H+u$, where $u$ is a smooth admissible potential,  meaning that 
$$
u(te_1)=0\quad,\quad  du(te_1)=0 \quad \forall t\in [0, \delta].
$$
 If $u$ is admissible and $H$ satisfies the conclusions of Theorem \ref{thm-nf}, then so does $H+u$, hence
the linearized transition maps $L_u(t)$ associated to $H+u$ are described as follows:

\begin{lemma}\label{lem-ce}
If $H:T^*\Rm^{d+1}\to \mathbb{R}$ satisfies the conclusions of Theorem \ref{thm-nf}, then for an admissible potential $u$,
the linearized transition maps  $L_u(t)$  associated to $H+u$ satisfy
\begin{equation}\label{eq: 4.3}
\dot L_u(t)=Y_u(t)L_u(t).    
\end{equation}
where $Y_u(t)=\mathbb{J}\partial^2_{\hat{x}^2}(H+u)(t e_1,0_n)=Y(t)+W_u(t)$, 
$$
W_u(t)=\begin{bmatrix}
0 & 0 \\
-\partial_{\hat q^2}^2 u(te_1) & 0
\end{bmatrix}= \begin{bmatrix}
0 & 0 \\
B_u(t) & 0
\end{bmatrix}\quad,\quad
B_u(t)=-\partial_{\hat q^2}^2 u(te_1).
$$
\end{lemma}

In order to prove Theorem \ref{thm-pert} we need to understand to what extent the linearized transition map $L_u(\delta)$ can be chosen by choosing the admissible potential $u$. 
Note that any compactly supported smooth curve $B(t):]0, \delta[\rightarrow \mS(d)$ 
can be obtained from an admissible potential, in the sense that there exists an admissible potential $u$ such that $B(t)=B_u(t)$.
So we are reduced  to the following non-autonomous bilinear control problem:
\begin{equation}\label{eq-contr}
\dot L_B(t)=Y(t)L_B(t)+ 
\begin{bmatrix}0&0\\B(t)&0
\end{bmatrix}
 L_B(t),
\end{equation}
where $Y(t)$ is a given curve of Hamiltonian matrices taking the specific form 
$$Y(t)=
 \begin{bmatrix}
	0 & D \\
	-K(t) & 0
\end{bmatrix}
$$ 
 and  $B(t)$ is a control taking values in $\mathcal{S}(d)$.
Recall that a matrix $M$ is called Hamiltonian if $\mathbb{J}M$ is symmetric. We denote by $\sp(2d)$ the set of Hamiltonian
matrices of size $2d\times 2d$. This Lie algebra is the tangent space at the identity of the group $Sp(2d)$ of symplectic matrices.
On this control problem, we can adapt  \cite{RR} to our setting and get:

\begin{proposition}\label{prop-lcp}
	There exists a nowhere dense closed set $\mK_D\subset \mS(d)$, which depends on the diagonal matrix $D$, such that  the differential at $B=0$ of the map 
	$$
	C^{\infty}_c(]0,s[, \mS(d))\ni B\longmapsto L_B(s)\in Sp(2d)
	$$
	is onto for each $s\in ]0, \delta[$ provided $K(0)\not\in \mK_D$, 
	where $C^{\infty}_c(]0,s[, \mS(d))$ is the space of compactly supported smooth curves,
	and $L_B(t)$ is the solution starting with initial condition $L_B(0)=0$ of the differential equation (\ref{eq-contr}).
\end{proposition}

Assuming this Proposition, we finish the proof of the Theorem.	

\noindent 
\textit{Proof of Theorem \ref{thm-pert}.}
Let $U$ be an open set in the space $C^{\infty}_{\theta}$ of admissible potentials.
There exists $u_0\in U$ such that  the corresponding  
$B_0(t)=\partial^2_{\hat q\hat q}u_0(te_1,0)$ satisfies $K(0)+B_0(0)\in \mK_D$.

Proposition \ref{prop-lcp} (applied to $Y+W_{u_0}$) implies the existence of a finite dimensional linear subspace 
$F\subset C^{\infty}_c(]0,\delta[, \mS(d))$ such that the restriction 
$$
F\ni B\mapsto L_{B_0+B}(\delta)\in Sp(2d)
$$
is a $C^1$ submersion at  $B=0$.
We denote  $\mL(B):= L_{B_0+B}(\delta)$ this map.
There is a finite dimensional subspace $E$ of smooth admissible potentials, the support of which
 do not intersect the remaining part $\pi \circ \theta ([\delta, T])$ of the  projected orbit, and such that 
the map 
$$
E\ni u\longmapsto B_u=-\partial^2_{\hat q \hat q}u \in F
$$
is a linear isomorphism from $E$ to $F$.

The restricted linearized Poincaré return map associated to the section $\{q_1=0\}$ for the Hamiltonian $H+u_0+u$ 
is the product $O\mL(B_u)$, where $O$ is the outer linearized transition map 
from the section  $\{q_1=\delta\}$ to the section  $\{q_1=0\}$ associated to the Hamiltonian $H+u_0$ (or equivalently to $H+u_0+u$).
Since $O$ does not depend on $u$, the  map $u \mapsto O\mL(B_u)$ is a submersion  in a neighborhood of $0$ in $E$, which implies that the image of $U$ contains an open set.
\qed

\textit{
Proof of  Proposition \ref{prop-lcp}.}
Let us first recall some general theory, following \cite{RR}.
We consider a smooth curve $Y(t)$ of Hamiltonian matrices, and the bilinear control problem
$$
L_B'(t)=Y(t)L_B(t)+W(B(t))L_B(t),
$$	
where $W$ is a linear map from a finite dimensional vector space $E$ to $\sp(2d)$ and $B(t)$ is a control taking values in  $E$.

For each fixed $B\in E$, we define the following sequence of curves of matrices:
 \begin{align*}
 &W_0(t,B)\equiv W(B) \quad, \quad W_1(t,B)= [W_0(B),Y(t)],\\
 &W_{i+i}(t,B)=\dot W_i(t,B)+[W_i(t,B), Y(t)],
 \end{align*}
 where $[A,B]$ is the commutation bracket $AB-BA$.
 Then, we consider the subspace 
 $$
 W_*E:=Vect\{W_i(0,B), i\geq 0, B\in E\}.
$$ 
The following general result is for example Proposition 2.1 in \cite{RR}:

\begin{proposition}
	If $W_*E=\sp(2d)$,  then for each $s>0$,  the differential at $B=0$ of the map  
	$$
	C^{\infty}_c(]0,s[, E)\ni B\longmapsto L_B(s)\in Sp(2d)
	$$
	is onto. Here $L_B(s)$ is the solution at time $s$ of the controlled  differential equation for the given control  $B(t)$, with the initial condition $L_B(0)=I$.
		\end{proposition}

	We now apply this general result to our situation of interest, where $E=\mS(d)$ and 
	$W(B)=\begin{bmatrix}
	0 & 0 \\
	B & 0
	\end{bmatrix}$, 
	$Y(t)=\begin{bmatrix}
	0 & D \\
	-K(t) & 0
	\end{bmatrix}
	$
	with some fixed diagonal matrix $D$ having diagonal elements equal to $\pm 1$.
	In order to deduce Proposition \ref{prop-lcp}, it is sufficient to check that the assumption $W_*E=\sp(2d)$ is satisfied
	under an appropriate assumption on  $K(0)$.
		We compute:
\begin{align*}
W_0(t,B)&= W= 
\begin{bmatrix}
0 & 0 \\
B & 0
\end{bmatrix}, \\
W_1(t,B)&= 
\begin{bmatrix}
-DB & 0 \\
0 & BD
\end{bmatrix}, \\ 
W_2(t,B)&= 
\begin{bmatrix}
0 & -2D B D \\
-B D K(t) -  K(t)D B & 0
\end{bmatrix}, \\
W_3(t,B)&= 
\begin{bmatrix}
3DBDK(t)+DK(t)DB & 0 \\ 
* & -BDK(t)D-3K(t)DBD
\end{bmatrix}, 
\end{align*}
where the $*$ block in $W_3$ is the derivative of the corresponding block in $W_2$.
We see that the matrices 
$$
W_0(0,B), B\in \mS(d);\quad W_1(0,B), B\in \mS(d); \quad  W_2(0,B), B\in \mS(d)$$
 generate the space of matrices 
of the form 
$\begin{bmatrix}
	m_1 & m_2 \\ m_3 & -m_1^T
	\end{bmatrix}$,
with $m_1, m_2$ and $m_3$ symmetric.
Recall that $\sp(2d)$ is the space of all matrices of that form with $m_2$ and $m_3$ symmetric, and all $m_1\in \mM(d)$.
In order that $W_*E=\sp(2d)$, it is sufficient that the matrices $3BDK(0)+K(0)DB$ generate a complement of $\mS(d)$ in $\mM(d)$, and this is equivalent to requiring that 
their antisymmetric parts $BDK(0)-K(0)DB$ generate the whole  space $\mS^-(d)$ of antisymmetric matrices. 
According to the next Lemma, this holds provided $K(0)$ does not belongs to an appropriate closed set $\mK_D$ with empty interior. This 
ends the proof of Proposition \ref{prop-lcp}.
\qed

\begin{lemma}
Let $D$ be a fixed diagonal matrix with diagonal elements equal to $\pm 1$ and let	 $\mK_D$ be the set of symmetric matrices $K\in \mS(d)$ such that the map
	$$
	\mS(d)\ni B\longmapsto BDK-KDB \in \mS^-(d)
	$$
	is not onto. Then $\mK_D$ is  a strict algebraic submanifold of $\mS(d)$, its complement is thus a dense open set.
\end{lemma} 

In the case where $D=I$ is the identity, $\mK_I$ is the space of symmetric matrices which have a multiple eigenvalue, see \cite{Co10}.

\begin{proof}
	It is clear that $\mK_D\subset \mS(d)$ is an algebraic submanifold.
It is thus  enough to find one matrix $K$ such that the map $B\mapsto BDK-DKB$ is onto.
We claim that this conclusion holds if $K$ is diagonal with distinct positive diagonal elements.
In this case, $\Delta:=DK=KD$ is also diagonal with distinct diagonal elements $\lambda_i$, and we can compute
$(BDK-KDB)_{ij}=(B\Delta-\Delta B )_{ij}=B_{ij}(\lambda_j-\lambda_i)$, which implies the claim.
\end{proof}

\section{Parametric transversality}\label{sec-pt}

We deduce Theorem \ref{thm-g1} from Theorem \ref{thm-pert} using a variant of Abraham's parametric transversality principle, similarly to what is done in \cite{A82,O08}. 
Our implementation is slightly different and avoids the recurrence on  periods.
We denote by $\varphi(t,x,u)$ the image of $x\in T^*M$ by the time $t$ flow of $H+u$.

It is classical that $0$ is a non-degenerate energy level for generic $u$. We recall a proof.
Let $\mZ\subset T^*M\times C^{\infty}(M)$ be the set of pairs $(x,u)$ such that $d(H+u)(x)=0$ and $(H+u)(x)=0$. This is a closed set, and the projection 
of a closed set on the second factor is an $F_{\sigma}$ (a countable union of closed sets). We use here that $T^*M$ is a countable union of compact sets.
So the set of potentials for which the $0$ energy level is not regular is an $F_{\sigma}$. If $u_0$ is a potential, then Sard's Theorem implies that there exist arbitrarily small constants $a$ such that $0$ is a regular value of $H+u_0+a$.
 As a consequence,  the set of potentials for which the $0$ energy level is regular is dense. This ends the proof of this first step.

Next we introduce the subsets 
$$
\mN \subset \mP \subset  ]0, \infty[\times T^*M\times  C^{\infty}(M).
$$
The larger subset $\mP$ is the set of  triples $(s,x,u)$ such that the $(H+u)$-orbit of $x$ is periodic of  period $s$ 
and $(H+u)(x)=0$. Note that $s$ is not necessarily the minimal period, so that $\mP$ contains $]0,\infty[\times \mZ$. The subset $\mP$ is clearly closed.

The smaller subset $\mN$  is the set of triples $(s,x,u) \in \mP$ such that $s$ is the minimal period of $x$, such that  $x$ is a neat point of the periodic orbit, and such that $H$ (or equivalently $H+u$) 
is fiberwise isoenergetically non-degenerate at $x$. Note that $\mN$  is disjoint from  $]0,\infty[\times \mZ$.

\begin{lemma}\label{lem-open}
	The subset $\mN$ is open in $\mP$, hence locally closed in 
	$]0, \infty[\times T^*M\times  C^{\infty}(M)$.
	\end{lemma}
Recall that a set is called locally closed if it is the intersection of a closed and of an open set. Since the ambient 
topology is metrizable,  locally closed sets are $F_{\sigma}$. 

\begin{proof}
	We will prove that $\mP-\mN$ is closed in $\mP$.
We consider a sequence $(s_k,x_k,u_k)\rightarrow (s,x,u)$  of points of $\mP-\mN$
converging to ($s,x,u)\in \mP$, and prove that the limit does not belong to $\mN$.

Each of the points $(s_k,x_k,u_k)$ violate one of the conditions defining $\mN$ in $\mP$, and we can suppose by taking a subsequence 
that they all violate the same condition.

A first possibility is that $s_k$ is not the minimal period of $x_k$.
We denote by $S_k$ the minimal period of $x_k$, so that $s_k=i_kS_k$ for some integers $i_k\geq 2$. By taking a subsequence, we can assume that $i_k$ either is constant or converges to $+\infty$.
In the first case, denoting by $i\geq 2$ the constant value of the sequence $i_k$, we obtain that
$S_k\rightarrow s/i$, and that $x$ is $s/i$-periodic, so that $s$ is not the minimal period of $x$.
In the second case, we obtain that $(x,u)\in \mZ$. In both cases,  $(s,x,u)$ does not belong to $\mN$.

A second possibility is that  the points $x_k$ are not neat.
We denote by $Q_k(t)$ the projected $(H+u_k)$-orbit of $x_k$.
Since $0$ is not a neat time we can assume by taking a subsequence that either $\dot Q_k(0)=0$ for each $k$ or there exists  times $t_k\in ]-s_k/2,0[\cup]0, s_k/2]$ 
such that $Q_k(0)=Q_k(t_k)$. The first case immediately implies that $\dot Q(0)= 0$, where $Q(t)$ is the limit projected orbit.
In the second situation, if the sequence $t_k$ has an accumulation point $t\neq 0$  then $Q(t)=Q(0)$.
This implies that either $s$ is not the minimal period, or $0$ is not a neat time. In both cases, $(s,x,u)\not \in \mN$.
Otherwise   $t_k \rightarrow 0$.
Then the equation $Q_k(t_k)=Q_k(0)$ imply at the limit that $\dot Q(0)=0$, hence once again $0$ is not a neat time at the limit.

The last possibility is that $H$ is fiberwise isoenergetically 
degenerate at $x_k$, but then it also is at the limit point $x$.
\end{proof}

 We now denote by $\mN(\Upsilon)\subset \mN$ the subset of triples $(s,x,u)\in \mN$ such that the corresponding restricted linearized return map belongs to $\Upsilon$.
 The conclusion of Theorem \ref{thm-g1} can be expressed by saying that $\Pi(\mN(\Upsilon))$ is an $F_{\sigma}$ with empty interior in $C^{\infty}(M)$, where
 $$
 \Pi:]0, \infty[\times T^*M\times  C^{\infty}(M)\rightarrow C^{\infty}(M)
 $$
 is the projection. This is what we will now prove. 
 
 The  fibers $]0, \infty[\times T^*M$ of the projection $\Pi$ are $\sigma$-compact (they are countable unions of compact sets) hence the image of a closed set by this projection is an $F_{\sigma}$, hence the image of an $F_{\sigma}$ is an $F_{\sigma}$.
 We can apply this remark to the set $\mN(\Upsilon)$, which is an $F_{\sigma}$ in $\mN$, hence in 
 $]0, \infty[\times T^*M\times  C^{\infty}(M)$. So the projection $\Pi(\mN(\Upsilon))$ is an $F_{\sigma}$.

 As already noticed in earlier works, it is convenient to decompose the set $\mN(\Upsilon)$ in two parts.
 
 The first part is $\mN^d$, the set of elements of $\mN$ which are degenerate, meaning that $1$ is an eigenvalue of the linearized restricted return map. It is nothing that $\mN(\Upsilon_1)$, where $\Upsilon_1$ is the set of symplectic matrices having the eigenvalue $1$, and it is contained in $\mN(\Upsilon)$ since we assumed that
 $\Upsilon_1\subset \Upsilon$ (otherwise we could just replace  $\Upsilon$ by $\Upsilon \cup \Upsilon_1$).
 We denote by $\mN^r$ the complement of $\mN^d$, \textit{i.e.} the set of regular triples. So for $(s,x,u)\in \mN^r$, the $H+u$ orbit of $x$ has minimal period $s$ and is non-degenerate as an $s$-periodic orbit. Note that it might be degenerate as a $2s$-periodic orbit. 
 
 The second part of $\mN(\Upsilon)$ that we shall consider is
 $$
 \mN^r(\Upsilon):= \mN^r\cap \mN(\Upsilon).
 $$
It corresponds to the periodic orbits in $\mN$ which are regular and have a return map in $\Upsilon$.
 Obviously, $\mN(\Upsilon)$ is the disjoint union of $\mN^d$ and of $\mN^r(\Upsilon)$.

 We will  prove in  Proposition \ref{prop-fnd} and Proposition
 \ref{prop-pt} below that $\Pi(\mN^r(\Upsilon))$ and $\Pi(\mN^d)$ are $F_{\sigma}$ with empty interior, which implies that 
 $\Pi(\mN(\Upsilon))$ is an $F_{\sigma}$ with empty interior and proves Theorem \ref{thm-g1}.
 
 \begin{proposition}\label{prop-fnd}
 	The set $\Pi(\mN^r(\Upsilon))\subset C^{\infty}(M)$ is an $F_{\sigma}$ with empty interior.
 \end{proposition} 
 
 \begin{proof}
 	We start with the main principles of the proof.
 	
 	The triples $(s,x,u)\subset \mN$ correspond to periodic orbits of $H+u$, and we consider the map
 	$
 	L:\mN \rightarrow  Sp(2d)
 	$
 	which to each such triple associates the restricted linearized return map of the corresponding periodic orbit. This map is actually not well defined globally, because it involves local choices of sections and coordinates, but we will temporarily continue the discussion as if it was well-defined.
 	The map $L$ is continuous and, in view of Theorem \ref{thm-pert}, weakly open (the image of an open set contains an open set).
 	Then $\mN(\Upsilon)=L^{-1}(\Upsilon)$ is an $F_{\sigma}$ with empty interior in $\mN$.
 	Since $\mN^r$ is open in $\mN$, we deduce similarly that  $\mN^r(\Upsilon)$ is an $F_{\sigma}$ with empty interior in $\mN^r$.
 	
 	The continuous dependence of non-degenerate periodic orbits imply that the restriction of $\Pi$ to $\mN^r$ is a local homeomorphism onto an open subset of  $C^{\infty}(M)$.
 	This is the place when we take advantage of studying $\mN^r(\Upsilon)$ rather than the whole $\mN(\Upsilon)$.
 	This property allows to conclude easily, from the fact that $\mN^r(\Upsilon)$ is an $F_{\sigma}$ with empty interior, 
 	that its projection is also an $F_{\sigma}$ with empty interior.

 	With the above steps as guidelines, we now give more details.
 	Since $\mN^r \subset ]0, \infty[\times T^*M\times  C^{\infty}(M)$ is a separable metric space, it is enough to prove the property locally. More precisely it is enough to prove that
 	each point $(s_0,x_0,u_0)\in \mN^r$  has an open  neighborhood $\mN^r_{loc}$ such that $\Pi(\mN^r_{loc}\cap \mN^r(\Upsilon))$ is an $F_{\sigma}$ with 
 	empty interior. 
 	Then, we can cover $\mN^r$ by countably many open subsets $\mN^r_{loc}$ having this property, hence $\Pi(\mN^r)$ is a countable union of $F_{\sigma}$ with empty interior, so by the Baire property, it is an $F_{\sigma}$ with empty interior.
 	
 	We now fix $(s_0,x_0,u_0)\in \mN^r$ and prove that $\Pi(\mN^r_{loc}\cap \mN^r(\Upsilon))$ is an $F_{\sigma}$ with 
 	empty interior provided $\mN^r_{loc}$ is a sufficiently small open neighborhood of $(s_0,x_0,u_0)$.
 	We can assume without loss of generality that $u_0=0$, and work in the local coordinates near $x_0$ given by Theorem \ref{thm-nf} (it is actually not necessary here to have such a specific normal form). In these coordinates, $x_0=0$.
	If $\mN^r_{loc}$ is sufficiently small, then for each $(\underline s,\underline x,\underline u)\in \mN^r_{loc}$, the point $\underline x$ belongs to the domain of local coordinates and we denote by $(\underline q_1,\underline p_1, \underline {\hat x})$ its coordinates. Still assuming that $\mN^r_{loc}$ is small enough,  we can take $\hat x$ as local symplectic coordinates of the  section 
	$\{q_1=\underline q_1, H+\underline u =0\}\subset T^*M$ near the point $\underline x$.
	We define  the matrix $L(\underline s, \underline x, \underline u)\in Sp(2d)$ as the restricted linearized return map
 	of the $H+\underline u$ orbit of $\underline x$ associated to the section
 	$\{q_1=\underline q_1, H+\underline u =0\}$, expressed in coordinates $\hat x$. The map 
 	$$
 	L:\mN^r_{loc}\longrightarrow Sp(2d)
 	$$
 	is continuous an weakly open, hence 
 	$$
 	\mN^r_{loc}(\Upsilon):= \mN^r_{loc}\cap \mN(\Upsilon)= L^{-1}(\Upsilon)
 	$$
 	is an $F_{\sigma}$ with empty interior in $\mN^r_{loc}$. By the continuous dependence of non-degenerate periodic orbits  on parameters, we can assume by further reducing $\mN^r_{loc}$ that the restriction of the projection $\Pi$ to this open subset is a homeomorphism onto an open subset $C^{\infty}_{loc}$ of $C^{\infty}(M)$. Then $\Pi(\mN^r_{loc}(\Upsilon))$ is an $F_{\sigma}$ with empty interior in  $C^{\infty}_{loc}$, hence in $C^{\infty}(M)$.
 	\end{proof}

 \begin{proposition}\label{prop-pt}
 	The set $\Pi(\mN^d)\subset C^{\infty}(M)$ is an $F_{\sigma}$ with empty interior.
 \end{proposition}
 
 \begin{proof}
 Since $]0, \infty[\times T^*M\times  C^{\infty}(M)$ is a separable metric space, it is enough to prove the property locally. More precisely it is enough to prove that
 each point $(s_0,x_0,u_0)\in \mN^d$  has an open  neighborhood $\mN^d_{loc}$ such that $\Pi(\mN^d_{loc})$ is nowhere dense. Note that $\mN^d_{loc}$ is locally closed, hence its projection is an $F_{\sigma}$.
 
 We can assume without loss of generality that $u_0=0$, and work in the local coordinates near $x_0$ given by Theorem \ref{thm-nf}. In these coordinates, $x_0=0$.
 We consider the transverse section $\{q_1=0\}$, and its intersection with the energy surface $\{H=0\}$. The tangent space at $x=0$ to this energy surface is $\{p_1=0\}$, and this implies that $\hat x=(\hat q, \hat p)$ are symplectic local  coordinates of
 the restricted section 
 $$\Lambda (u):= \{(q,p)\in \Rm^{2(d+1)} : \; q_1=0, (H+u)(q,p)=0\}$$
  provided $u$ belongs to a sufficiently small open neighborhood $C_{loc}^{\infty}$ of $0$. We denote by $x(\hat x,u)$ the point of $\Lambda(u)$ which has coordinate $\hat x$,
 by 
 $$
 \tau(\hat x,u): \Rm^{2d}_{loc}\times C_{loc}^{\infty}\longrightarrow \Rm
 $$
 the first return time of the point $x(\hat x,u)$ to the section $\Lambda(u)$. It is well defined and smooth provided $\Rm^{2d}_{loc}$ is a sufficiently small open neighborhood of $0$ in $\Rm^{2d}$.
 We also define the return map
 $$\psi (\hat x,u ):\Rm^{2d}_{loc}\times C_{loc}^{\infty}\longrightarrow \Rm^{2d}
 $$
  expressed in coordinates, meaning that $\psi(\hat x, u)$ is the $\hat x$ coordinates of the point 
 $\varphi(\tau (\hat x,u), x(\hat x,u),u)$.
 Let
  $$ \mY\subset \Rm^{2d}_{loc}\times C_{loc}^{\infty}
  $$
 be the set of solutions of the equation $\psi(\hat x, u)=\hat x$, and let $\mY^d\subset\mY$ be the set of degenerate solutions of this equation, meaning those such that  $\partial_{\hat x}\psi(\hat x, u)$ has the eigenvalue $1$.
 Given $(\hat x, u)\in \Rm^{2d}_{loc}\times C_{loc}^{\infty}$, we have    $(\hat x,u)\subset \mY$ if and only if 
 $(\tau(\hat x,u),x(\hat x, u),u)\subset \mP$, and this is also equivalent to $(\tau(\hat x,u),x(\hat x, u),u)\subset \mN$ (because $\mN$ is open in $\mP$). We have    $(\hat x,u)\subset \mY^d$ if and only if $(\tau(\hat x,u),x(\hat x, u),u)\subset \mN^d$.
 The set 
 $$
 \mN_{loc} := \{(\tau(\hat x, u), \varphi(t, x(\hat x, u),u), u): t\in \Rm, (\hat x,u)\in \mY\}
 $$
 is an open neighborhood of $(s_0,0,0)$ in $\mN$, which has the property that 
 $\Pi(\mN_{loc})=\Pi(\mY)$ 
 and, setting $ \mN^d_{loc}:=\mN^d\cap \mN_{loc}$,
 $$
 \Pi(\mN^d_{loc})=\Pi(\mY^d)
 $$
 where we still use $\Pi$ for the projection on the second factor in the product   $\Rm^{2d}_{loc}\times C_{loc}^{\infty}$.
 
 So we are reduced to proving that $\Pi(\mY^d)$ is nowhere dense.
 The general idea is that $\mY$ is a submanifold in $\Rm^{2d}_{loc}\times C_{loc}^{\infty}$ and that 
 $\mY^d$ is the set of regular points of $\Pi$ on this manifold, which allows to conclude using Sard's Theorem.
 In order to avoid manipulating differential calculus on the Fréchet space  $C^{\infty}(M)$, we restrict ourselves to a finite
 dimensional subspace.
 
 \begin{lemma}\label{lem-c2}
 	For any neighborhood $U$ of $0$ in $\Rm^{d+1}$, 
 	there exists a finite dimensional subspace $E\subset C^{\infty}(U)$ formed by potentials compactly supported inside $U$ and null on the orbit $\Rm\times \{0\}$, such that 
 	$\partial_u \psi(0,0)$ sends $E$ onto $\Rm^{2d}$. 
 	\end{lemma}
 
 We finish the proof of Proposition \ref{prop-pt} assuming the Lemma.
 By continuity, the equality $\partial_u \psi(\hat x,u)\cdot E=\Rm^{2d}$ holds for all
  $(\hat x, u)\in \Rm^{2d}_{loc}\times C^{\infty}_{loc}$ (we decrease the neighborhoods if necessary).
 Then for each $v\in C^{\infty}_{loc}$, the map 
 $$\Psi: \Rm^{2d}_{loc}\times E_{loc} \ni (\hat x,u) \mapsto \Psi(\hat x,u):=\psi(\hat x,v+u)-\hat x \in \Rm^{2d}$$
  is a submersion, provided $E_{loc}$ is a sufficiently small neighborhood of $0$ in $E$.
    As a consequence, the set  $N:=\Psi^{-1}(0)$ is a submanifold, and the points $(\hat x,u)\in N$ such that $\partial_{\hat x} \Psi(\hat x,u)$
  is not invertible are the singular  points of the projection $\Pi_{|N}$ (this follows from elementary considerations in finite dimensional linear algebra).
  If $v+u$ belongs to $\Pi(\mY^d)$, then there exists a point $\hat x\in \Rm^{2d}_{loc}$ such that 
  $(\hat x, v+u)\in \mY^d$, which implies that $(\hat x,u)\in N$ is a critical point of $\Pi_{|N}$.
  So $u$ has to be a critical value of $\Pi_{|N}$. By Sard's Theorem, there exist regular values of $\Pi_{|N}$
  arbitrarily close to $0$, and this implies that $v$ does not belong to the interior of $\Pi(\mY^d)$.
  Since this holds for all $v\in C^{\infty}_{loc}$, we conclude that $\Pi(\mY^d)$ has empty interior.
  This ends the proof of Proposition \ref{prop-pt}.
  \end{proof}

 \textit{Proof of Lemma \ref{lem-c2}.}
 As for  the proof of Lemma 2 in \cite{A82}, the key element is that the matrix $D=\partial^2_{\hat p \hat p}H(0)$ is invertible, which is precisely the expression in coordinates of fiberwise isoenergetic non-degeneracy.
 As in the proof of Lemma \ref{lem-el}, we consider the vectorfield $Z(x,u):=\mathbb{J} \partial_x (H+u )(x)/\partial_{p_1}H(x)$ and denote by 
 $\varphi_Z(t,x,u)$ its flow.  
 The curve $\omega(t):=\partial_u \varphi_Z(t,0,0)\cdot u$  solves the differential equation 
 $$
 \omega'(t)=\partial_x Z (te_1,0,0)\cdot \omega(t)+\partial_uZ(te_1,0,0)\cdot u.
 $$
 This equation can be obtained by differentiating the differential equation
 $$
 \partial_t \varphi_Z(t,x,u)= Z(\varphi_Z(t,x,u),u)
 $$
 with respect to $u$ along the orbit $\varphi_Z(t,0,0)=(te_1,0)$.
 Note that $\omega(t)$ depends (linearly) on $u$ but we do not  explicit this dependence to simplify notations.

The matrix $\partial_x Z (te_11,0,0)$ was computed in the proof of Lemma \ref{lem-el}, it has the block form
$$
\partial_x Z(te_1,0,0)=
\begin{bmatrix}
	0&0\\
	* & Y(t)
\end{bmatrix}
$$
where the blocs are associated to the decomposition $x=(x_1, \hat x)$.
As to the vector $\partial_uZ(te_1,0,0)\cdot u$, recall that
 $Z(x,u)=f(x)X_H(x,u)$,
 where $f(x)=1/\partial_{p_1}H(x,u)$. The function $f$ does not depend on $u$ and satisfies $f(te_1,0)=1$,
 hence  $\partial_u Z(te_1,0,0)\cdot u=\partial_uX_H(te_1,0,0)\cdot u$. If the potential $u$ vanishes on the orbit,
  we obtain, in coordinates 
 $(x_1, \hat q, \hat p)$:
$$
 \partial_uZ(te_1,0,0)\cdot u= (0,0, -\partial_{\hat q} u(te_1)).
 $$
 So the solution emanating from the initial condition 
 $\omega(0)=0$ is  $\omega(t)=(0, \hat \omega(t))$, 
 where $\hat \omega(t)$ is the solution emanating from $\hat \omega(0)=0$ of the equation
 \begin{equation}\label{eq-c2}
 \hat \omega'(t)=Y(t)\hat \omega(t)+\hat w_u(t),
 \end{equation}
 with
 $Y(t)=\begin{bmatrix} 0&D\\-K(t)&0\end{bmatrix}
 $ 
 and 
 $\hat w_u(t)=\begin{bmatrix} 0\\-\partial_{\hat q}u(te_1)\end{bmatrix}$.
 We denote by $\psi(t,\hat x,u)$ the restricted transition map between the sections $\{q_1=0\}$ and $\{q_1=t\}$, so that $\psi(t, \hat x, u)$ is  the $\hat x$-coordinate of $\varphi_Z(t,x(\hat x,u),u)$ where $x(\hat x,u)$ is the point of $\Lambda(u)$ with coordinate $\hat x$. As a consequence,
 $$\partial_u\psi (t,0,0)\cdot u=\hat \omega(t)
 $$ 
is determined by the differential equation (\ref{eq-c2}), and by the initial condition $\hat \omega(0)=0$.
 
We are thus once more reduced to a linear controlled differential equation   :
$$
\hat \omega'(t)=Y(t)\hat \omega(t)+b(t)
$$
where $b(t)$ is a control with values in $\{0\}\times \Rm^d$.
Denoting by $L(t)$ the curve of matrices such that $L(0)=I$ and $L'(t)=Y(t)L(t)$ and by $y(t)=L^{-1}(t)\hat \omega(t)$, this controlled equation is reduced to :
$$
y'(t)=L^{-1}(t) b(t).
$$
We fix  $\sigma >0$ and consider a control of the form 
$b(t)=\delta(t)\alpha$, with $\alpha$ fixed in $\{0\}\times \Rm^d$ and where $\delta(t)$ is a smooth approximation of the dirac at $0$ supported in $]0,\sigma[$.
We see that the corresponding $y(s)$ is an approximation of $\alpha$.
Now for a control $b(t)=\delta'(t)\beta$, we have 
$$
y'(t)=(L^{-1}\delta)'\beta -(L^{-1})'\delta \beta
=(L^{-1}\delta)'\beta +L^{-1}Y\delta \beta
$$
hence $y(\sigma)\approx Y(0)\beta$. Since $D$ is invertible, the vectors $e_i+Y(0)e_j, 1\leq i,j \leq d$ span $\Rm^{2d}$
(where $e_i$ is the standard base of $\Rm^d$).
So if  $\delta$ is a sufficiently good approximation of the Dirac at $0$, then
the values $y(\sigma)$ of the solutions of the controlled differential equation corresponding to the controls
$\delta(t)e_i +\delta'(t)e_j, 1\leq i,j\leq d$, span $\Rm^{2d}$.
Now for each $i,j$, we consider a potential $u_{ij}$ compactly supported in $U$, null on the orbit, and such that 
$\hat w_{u_{ij}}= \delta(t)e_i +\delta'(t)e_j, 1\leq i,j\leq d$
and define $E$ as the space generated by these potentials. We can moreover assume that the supports are   disjoint from the remaining part $Q([\sigma,s_0])$ of the projected orbit.
The linear map
$\partial_u\psi(\sigma,0,0): u \mapsto \hat \omega(\sigma)$
is then onto.

To end the proof, we denote by $G(\sigma,\hat x)$ the restricted section map from $\{q_1=\sigma \}$ to $\{q_1=0\}$ along the  periodic orbit. This map does not depend on $u\in E$ 
since the supports have been taken disjoint from $Q([\sigma,s_0])$. The restricted Poincaré return map satisfies
$$\psi(\hat x,u)=G(\sigma,\psi(\sigma, \hat x, u))
$$
and $\partial_u\psi(0,0)=\partial_{\hat x} G (\sigma ,0)\circ \partial_u \psi(\sigma,0,0)$
is onto.
\qed

\section{Orbits inside a submanifold.}\label{sec-o}

We prove Theorem \ref{thm-g2}.
It is enough to prove the statement under the assumption that $\Sigma\subset T^*M$ is a submanifold of positive codimension transverse to the vertical.
Let  $\mR(\epsilon)$ be the set of couples $(x,u)\in T^*M \times C^{\infty}(M)$ such that 
$\partial_pH(x)\neq 0$ and such that $\varphi([-\epsilon,\epsilon]\times \{x\}\times \{u\})\subset \Sigma$. We want to prove that 
$\Pi(\cup_{\epsilon >0}\mR(\epsilon))=\cup _{\epsilon >0}\Pi(\mR(\epsilon))$ is contained in an $F_{\sigma}$ with empty interior.
Since $\cup _{\epsilon>0}\Pi(\mR(\epsilon))$ is  equal to the countable union $ \cup _{n\in \Nm^*} \Pi (\mR(1/n))$, 
it is enough to prove that each of the sets $\Pi(\mR(\epsilon)), \epsilon >0$,
is contained in an $F_{\sigma}$ with empty interior. Moreover, it is enough to do so locally.
More precisely, it is enough to prove, 
 given $(x_0,u_0)\in \mR(\epsilon)$, the existence of   open neighborhoods $T^*M_{loc}$ of $x_0$ in $T^*M$ and $C^{\infty}_{loc}$ of $u_0$ in $C^{\infty}(M)$ such that
 $\Pi\big((T^*M_{loc}\times C^{\infty}_{loc})\cap \mR(\epsilon)\big)$ is 
 contained in an $F_{\sigma}$ with empty interior.

We consider many small times $0<\sigma_0<\sigma_1<\sigma_2<\cdots <\sigma_k <\sigma_{k+1}<\epsilon$, with $k>2d+2$  and define the map
$$
\Phi: T^*M\times C^{\infty}(M) \ni (x,u)\mapsto (\varphi(\sigma_1,x,u), \ldots , \varphi(\sigma_k,x,u))\in (T^*M)^k,
$$
so that $\mR(\epsilon)\subset \Phi^{-1}(\Sigma^k)$.
\begin{lemma}\label{lem-vert}
	For each $(x_0,u_0)$ such that $\partial_pH(x_0)\neq 0$, the times $\sigma_i$ can be chosen such that 
there exists a finite dimensional subspace $E\subset C^{\infty}(M) $ for which  the map
$$
T^*M\times E\ni(x,u)\mapsto \Phi(x,u_0+u)
$$  is transverse to $\Sigma^k$ at the point $(x_0,0)$.
	\end{lemma}

We fix $(x_0, u_0)\in \mR(\epsilon)$.
Assuming the Lemma,  there exist open neighborhoods $T^*M_{loc}$, $C^{\infty}_{loc}$ and $E_{loc}$  of $x_0$, $u_0$ and $0$ in $T^*M$, $C^\infty $ and $E$ such that the  map
$$
\tilde \Phi : T^*M_ {loc}\times E_{loc}\ni(x,u)\mapsto \Phi(x,u_1+u)
$$
is transverse to $\Sigma^k$ for each fixed $u_1\in C^{\infty}_{loc}$. 
We deduce that  $\tilde\Phi^{-1}(\Sigma^k)$ is a submanifold in $T^*M_{loc}\times E_{loc}$, of codimension equal to the codimension of $\Sigma^k$ in 
$(T^*M)^k$, and this codimension is at least $k$. Since $k>2d+2=\dim (T^*M)$,
the dimension of $\tilde \Phi^{-1}(\Sigma^k)$  is smaller than the dimension of $E$, hence the projection of $\tilde \Phi^{-1}(\Sigma^k)$ on $E$ is nowhere dense.
Setting 
$$\Phi_{loc}^{-1}(\Sigma^k):= \Phi^{-1}(\Sigma^k)\cap (T^*M_{loc}\times  C^{\infty}_{loc}),
$$
 we deduce that $u_1$ is not in the interior of $\Pi(\Phi_{loc}^{-1}(\Sigma^k))$, and since this holds for every $u_1\in C^{\infty}_{loc}$, we deduce that $\Pi(\Phi_{loc}^{-1}(\Sigma^k))$ has empty interior.
Moreover, $\Phi_{loc}^{-1}(\Sigma^k)$ is locally closed in $T^*M\times C^{\infty}(M)$,
hence it is an $F_{\sigma}$, so that its projection is also an $F_{\sigma}$.
This $F_{\sigma}$ with empty interior contains the projection of 
$(T^*M_{loc}\times C^{\infty}_{loc})\cap \mR(\epsilon)$, which concludes the proof of Theorem \ref{thm-g2}.
 \qed

\textit{Proof of Lemma \ref{lem-vert}.}
There is no loss of generality to assume that $u_0=0$.
Since $H$ is not assumed fiberwise isoenergetically non-degenerate at $x_0$, Theorem \ref{thm-nf} can't be applied at $x_0$. However there exist local  coordinates such that $\varphi(t,x_0,0)=(te_1,0)$
for small $|t|<\delta$ for some $\delta>0$. 
A look at the  proof of Theorem \ref{thm-nf} shows that such coordinates exist under the only assumption that $\partial_pH(x_0)\neq 0$.
We assume that $\sigma_i\in ]0,\delta[$, in particular the projected orbit is one to one
on $[0, \sigma_{k+1}]$.

Recall that, for a given potential $u$, the curve $t\mapsto d(t):=\partial_u\varphi(t,0,0)\cdot u$ is determined by the non-homogeneous linear equation 
$$
d'(t)=\xi (t) d(t)+(0,- du(te_1))
$$
and by the initial condition $d(0)=0$, where $\xi(t)=\mathbb{J}\partial^2_{xx}H(te_1,0)$ is the linearized equation.
Unlike in earlier sections of the paper, we have no information on $\xi(t)$.
We denote by $\Xi_s^t$ the family of solutions of the equation 
$$
\partial_t \Xi_s^t =\xi(t)\Xi_s^t
$$
with $\Xi_s^s=I$. Observe that $\Xi_s^t=\partial_x\varphi(t-s,se_1,0)$.

Let $\delta_i(t)$ be smooth approximations of the Dirac function at time $\sigma_i$, supported in $]\sigma_{i-1}, \sigma_i[$.
For each $i$ and $j$,
there exists a smooth potential 
$u_{i,j}$ such that
$$
du_{i,j}(te_1)=-e_j \delta_i(t),
$$
where $(e_j)$ is the standard base of $\Rm^{d+1}$.
The differential
$\partial_u \varphi(\sigma_i,0,0)\cdot u_{i,j}$ is approximately the vertical vector $l_j:=(0,e_j)\in \Rm^{d+1}\times \Rm^{d+1}$.
Then for $i'\geq i$, 
$$
\partial_u \varphi(\sigma_{i'},0,0)\cdot u_{i,j}\approx \Xi_{\sigma_i}^{\sigma_{i'}} l_j.
$$
As a consequence, we have 
$$
\partial_u\Phi(0,0)\cdot u_{i,j}\approx \eta_{i,j}:=(0, \ldots, 0, l_j,\Xi_{\sigma_i}^{\sigma_{i+1}}l_j,\ldots , \Xi_{\sigma_i}^{\sigma_k}l_j)\in (\Rm^{2d+2})^k.
$$
The vectors $\eta_{i,j}$ span the product of verticals $(\{0\}\times \Rm^{d+1})^k$ which is a subspace  transverse to $\Sigma^k$.
By taking for $\delta_i$ sufficiently good approximations of Dirac functions, we can make the vectors  $\partial_u\Phi(0,0)\cdot u_{i,j}$ 
as close as we want to $\eta_{i,j}$, and then  they also span a vector subspace transverse to $\Sigma^k$. 
As a consequence the $k(d+1)$-dimensional vector space $E$ generated by the potentials $u_{i,j}$ has the property that 
$
\partial_u\Phi(0,0)
$
sends $E$ to a subspace transverse to $\Sigma^k$.
\qed

%

\end{document}